\documentclass[a4paper,reqno]{amsart}
\usepackage{amssymb, mathrsfs, amscd, amsaddr} 

\theoremstyle{plain}
\newtheorem{thm}{Theorem}
\newtheorem{lem}{Lemma}
\newtheorem{cor}{Corollary}

\theoremstyle{remark}
\newtheorem{rem}{Remark}

\usepackage[english]{babel}
\usepackage[autostyle, english = american]{csquotes}
\MakeOuterQuote{"}

\usepackage{hyperref}
\hypersetup{
    colorlinks,%
    citecolor=red,%
    filecolor=purple,%
    linkcolor=blue,%
    urlcolor=blue
}

\newcommand\numberthis{\addtocounter{equation}{1}\tag{\theequation}}

\renewcommand{\v}[1]{\ensuremath{\mathbf{#1}}} 
\newcommand{\gv}[1]{\pmb{#1}} 
 
\newcommand{\mc}[1]{\mathcal{#1}} 
\newcommand{\mb}[1]{\mathbb{#1}} 
 
\newcommand{\mr}[1]{\mathrm{#1}}

\newcommand{\R}{\mb{R}}
\newcommand{\C}{\mb{C}}

\newcommand{\cF}{\mc{F}}

\newcommand{\dd}{\mr{d}}

\newcommand{\gvsig}{\gv{\sigma}}

\newcommand{\gvpi}{\gv{\pi}}

\DeclareMathOperator{\curl}{curl}
\DeclareMathOperator{\diver}{div}
\DeclareMathOperator{\re}{Re}

\allowdisplaybreaks
\begin{document}
\title[]{Non-linear Schr\"{o}dinger Equation in a Uniform Magnetic Field}
\author{T. F. Kieffer and M. Loss}
\address{School of Mathematics, Georgia Institute of Technology, Atlanta GA, USA.}

\date{October 24, 2020}

\begin{abstract}
The aim of this paper is to study, in dimensions $2$ and $3$, the pure-power non-linear Schr\"{o}dinger equation with an external uniform magnetic field included. In particular, we derive a general criteria on the initial data $\psi_0 \in H_{\!\v{A}}^1$ and the power of the non-linearity so that the corresponding solution blows up in finite time, and we show that the time for blow up to occur decreases as the strength of the magnetic field increases. In addition, we also discuss some observations about Strichartz estimates in $2$ dimensions for the Mehler kernel, as well as similar blow-up results for the non-linear Pauli equation. 
\end{abstract}

\maketitle
 \vskip .3 true in
\noindent
{\it We dedicate this work to Ari who has  influenced our research in many and fruitful ways.}




\section{Introduction}\label{sec:intro}

The importance of Virial Identities in the theory of partial differential equations can hardly be overstated.
They can be used for proving absence of bound states in Schr\"odinger operators (see \cite{RSIV}) and
the well known Pohozaev identity is a Virial Identity in disguise. What is maybe more surprising is that
a Virial Identity is an entry point to scattering theory and has lead to what is now called `Mourre Estimates'.
A very readable introduction can be found in \cite{CFKS1987}. Another application, and this is the topic of this paper, is the study of blow up of solutions for the non-linear Schr\"odinger equation
\begin{align}\label{eq:NLS}
\left\lbrace  \begin{array}{l}
i \partial_t \psi = - \Delta \psi + \mu | \psi |^{p-1} \psi \\
\psi (0) (\v{x}) = \psi_0 (\v{x}) ,
\end{array} \right. 
\end{align}
due to Glassey \cite{glassey1977}. This system has a conserved energy
\begin{align}\label{def:NLS_energy_nofield}
E [\psi (t)] = \| \nabla \psi (t) \|_{L^2 (\R^d)}^2 + \frac{2 \mu}{p+1} \| \psi (t) \|_{L^{p+1} (\R^d)}^{p+1} ,
\end{align}
and one computes with Glassey
\begin{align}\label{eq:glassey_virial}
\frac{1}{4} \frac{\dd^2}{\dd t^2} \langle \psi , |\v{x}|^2 \psi \rangle_{L^2 (\R^d)} = 2 \| \nabla \psi \|_{L^2 (\R^d)}^2 + \mu  d \frac{p-1}{p+1} \| \psi \|_{L^{p+1} (\R^d)}^{p+1} .
\end{align}
Thus, when $p \ge 1 + 4/d$ and $\mu<0$, then
\begin{align*}
\frac{1}{4} \frac{\dd^2}{\dd t^2} \langle \psi , |\v{x}|^2 \psi \rangle_{L^2 (\R^d)} \leq 2 E [\psi_0] .
\end{align*}
If the energy is strictly negative, then the function $  \langle \psi , |\v{x}|^2 \psi \rangle_{L^2 (\R^d)}(t)$ is strictly concave and must hit zero at some time $T$ provided that the solution were to exist up to that point in time. Thus, the solution ceases to exist after a finite time.

The initial value problem (\ref{eq:NLS}) has been studied extensively by many authors; see, for example, \cite[Chapter 3]{tao-dispersive2006} or \cite{Cazenave89} for detailed reviews. Local well-posedness in $H^1 (\R^d)$ of (\ref{eq:NLS}) for $1 < p < 1 + 4/(d-2)$ essentially follows from a fixed point argument utilizing the Strichartz estimates for the Schr\"{o}dinger unitary group $e^{i t \Delta}$. Such estimates read
\begin{align}\label{eq:strichartz_estimates}
\| e^{i t \Delta} f \|_{L^q_t L^r (\R \times \R^d)} \leq C \| f \|_{L^2 (\R^d)} ,
\end{align}
for all $f \in L^2 (\R^d)$ and all \textit{Schr\"{o}dinger admissible} exponents $(q,r)$ satisfying
\begin{align*}
q \in [2 , \infty], \hspace{1cm} \frac{2}{q} = d \left( \frac{1}{2} - \frac{1}{r} \right) , \hspace{1cm} (q,r,d) \neq (2 , \infty , 2) .
\end{align*}
Here $C > 0$ is a constant depending on the numbers $(q,r,d)$.  Global well-posedness for $p < 1+4/d$ follows along familiar lines using the energy conservation and the Gagliardo-Nirenberg inequalities.

The aim of this paper is to generalize the aformentioned blow-up result to the pure-power NLS equation with a uniform magnetic field included, referred to as the constant field non-linear magnetic Schr\"{o}dinger (NLMS) equation. In particular, we derive a general criteria on the initial data $\psi_0 \in H_{\!\v{A}}^1$ so that finite time blow-up occurs when $p \geq 1 + 4/d$. It is easy to see by a scaling argument that the length of the time interval of local existence is inversely proportional to the strength of the magnetic field  (see Theorem \ref{thm:blowup_NLMS}). Since Strichartz estimates are at the heart of some of the estimates for the nonlinear Schr\"odinger equation we add some simple and, we think, interesting remarks about Strichartz estimates in $2$ dimensions for the Mehler kernel (see Theorem \ref{thm:strichartz_identity}). We end the paper with a few simple remarks concerning blow-up results for the non-linear Pauli equation (see Theorem \ref{thm:local_wellposed_Pauli} and \ref{thm:blowup_NLP}). 

A resolution of the problem of finite time blow-up for the constant field NLMS equation was already claimed by G.~Ribeiro in 1990 \cite{ribeiro91} in $3$ dimensions, and subsequently generalized by A. Garcia in 2012 \cite{garcia2011}. However, we believe that our approach brings a new perspective to this problem as we observe that a crucial feature is the existence of an additional conserved quantity related to the angular momentum (see (\ref{def:blowup_functional_schr})). Using this additional conserved quantity, a more exact, as compared to the results of \cite{ribeiro91, garcia2011}, virial identity for the second time derivative of the expectation value of $|\v{x}|^2$ is derived (see (\ref{eq:constantfield_xsquare_gauge_independent})). This new identity yields a different sufficient condition on the initial data which guarantees finite time blow-up, and may also be solved exactly in $2$ dimensions with critical power for the non-linearity (see (\ref{eq:constantfield_explicitsolution})). 

The paper is organized as follows. In \S\ref{sec:NLMS} we introduce the NLMS equation, discussing local/global well-posedness and Strichartz estimates in \S\ref{sub:strichartz}, and study finite time blow-up in \S\ref{sub:blowup} and \S\ref{sub:virial}. We conclude with \S\ref{sec:NLP} which generalizes the results concerning the NLMS equation to the so-called non-linear Pauli equation.

{\bf Acknowledgment: }The authors are grateful to Jan-Philip Solovej for many discussions and for the hospitality at the University of Copenhagen. This work was partially funded by NSF grant DMS-1856645.

\section{Non-linear Magnetic Schr\"{o}dinger Equation}\label{sec:NLMS}

The Cauchy problem for the NLMS equation in $d \geq 2$ space dimensions\footnote{We only discuss (\ref{eq:NLMS}) in dimension $d \geq 2$ because in dimension $d = 1$ it is always possible to pass from $(\v{p} + \v{A})^2$ to the free Hamiltonian $- \Delta$ via a gauge transformation.} reads
\begin{align}\label{eq:NLMS}
\left\lbrace  \begin{array}{l}
i \partial_t \psi = (\v{p} + \v{A})^2 \psi + \mu | \psi|^{p-1} \psi \\
\psi (0 , \v{x}) = \psi_0 (\v{x}) .
\end{array} \right. 
\end{align}
where $\v{A}$ is the magnetic vector potential.
 We typically consider (\ref{eq:NLMS}) as an initial value problem in the Hilbert space
\begin{align*}
H_{\!\v{A}}^1 (\R^d ; \C) = \{ f \in L^2 (\R^d ; \C) ~ : ~ (\v{p} + \v{A}) f \in L^2 (\R^d ; \C^d) \} ,
\end{align*}
which is equipped with the norm $\| (\v{p} + \v{A}) f \|_{L^2 (\R^d)}$, and we will always assume the operator $(\v{p} + \v{A})^2$ is essentially self-adjoint on $L^2 (\R^d ; \C)$ with domain
\begin{align*}
H_{\!\v{A}}^2 (\R^d ; \C) = \{ f \in H_{\!\v{A}}^1 (\R^d ; \C) ~ : ~ (\v{p} + \v{A})^2 f \in L^2 (\R^d ; \C) \} .
\end{align*} 
If $\v{A} \in L^4(\R^d;\R^d)$ and $\diver{\v{A}} \in L^2(\R^d)$, then by the Leinfelder-Simader theorem \cite{LeiSim}, $(\v{p}+\v{A})^2$ is essentially self-adjoint on $C^\infty_c(\R^d)$.

The total energy associated with the NLMS equation (\ref{eq:NLMS}) is\footnote{The subscript "S" is placed on certain quantities to distinguish them from similar quantities that come up in the discussion of the non-linear Pauli equation in \S\ref{sec:NLP}.}
\begin{align}\label{def:NLMS_energy}
E_{\mr{S}} [\psi , \v{A}] (t) =  T_{\mr{S}} [\psi (t) , \v{A}] + \frac{2 \mu}{p+1} \| \psi (t) \|_{L^{p+1} (\R^d)}^{p+1} ,
\end{align}
where $T_{\mr{S}} [\psi , \v{A}] = \| (\v{p} + \v{A}) \psi \|_{L^2 (\R^d)}^2$. Note that $E_{\mr{S}} [\psi , \v{0}] = E [\psi]$, where on the right hand side appears (\ref{def:NLS_energy_nofield}). Often the magnetic vector potential $\v{A}$ is understood, and thus we will usually suppress the $\v{A}$-dependence of $E_{\mr{S}}$, simply writing $E_{\mr{S}} [\psi]$, and likewise for $T_{\mr{S}}$. It is straightforward to check via differentiation that, at least formally, $\| \psi (t) \|_{L^2 (\R^d)}^2$ and the total energy (\ref{def:NLMS_energy}) are conserved along the flow generated by (\ref{eq:NLMS}). 

\begin{rem}
It is common to apriori "fix a gauge" for the vector potential $\v{A}$. When a gauge is chosen, we will always pick the \textit{symmetric gauge} $\v{A} = B \v{x}_{\perp} / 2$, where
\begin{align*}
\v{x}_{\perp} = \left\lbrace \begin{array}{ll}
(- x_2 , x_1) , & d = 2 \\
(-x_2 , x_1 , 0) , & d = 3 .
\end{array} \right.
\end{align*}
However, unless otherwise specified, we will generally not fix the gauge for reasons that will become clear later.
\end{rem}

\subsection{Strichartz Estimates and Well-posedness}\label{sub:strichartz}

Local well-posedness for the NLMS equation (\ref{eq:NLMS}) with initial data in $H^1_{\!\v{A}} (\R^3)$ first appeared in \cite{cazenave_esteban_88}. There the authors consider more general non-linearities and external potentials, and also study the orbital stability of the ground state associated with (\ref{eq:NLMS}). Local well-posedness results for uniform magnetic fields in dimensions $d \geq 2$, and for more general magnetic fields, may be found in \cite[Theorem 2.2]{garcia2011}, and are proved using the Strichartz estimates of \cite{ancona_fanelli_08, ancona_fanelli_10}. For the sake of completeness, we state the local well-posedness result of \cite{cazenave_esteban_88} for the pure-power non-linearity case and no external potential.
\begin{thm}[Cazenave-Esteban, 1988 \cite{cazenave_esteban_88}]\label{thm:CE88}
Let $d = 3$, $\mu \in \R$, $p \in (1 , 5)$, and $\v{A} = B \v{x}_{\perp} / 2$. For all $\psi_0 \in H_{\!\v{A}}^1$ we have the following.
\begin{enumerate}
\item There exists a unique maximal solution $\psi \in C([0,T_*) , H_{\!\v{A}}^1) \cap C^1 ([0,T_*) , H_{\!\v{A}}^{-1})$ of (\ref{eq:NLMS}). If $T_* < \infty$, then $\| (\v{p} + \v{A}) \psi (t) \|_{L^2 (\R^d)} \rightarrow \infty$ as $t \uparrow T_*$.
\item The mapping $\psi_0 \mapsto T_* (\psi_0)$ is lower semi-continuous and, if $t \in [0 , T_* (\psi_0))$ and $( \phi_n )_{n \geq 1} \subset H^1_{\!\v{A}}$ converges to $\psi_0$ as $n \rightarrow \infty$, in $H_{\!\v{A}}^1$, then the corresponding sequence of solutions $(\psi_n)_{n \geq 1}$ to (\ref{eq:NLMS}) verify $\psi_n \rightarrow \psi$ as $n \rightarrow \infty$, in $C([0,t] , H_{\!\v{A}}^1)$.
\item If $\psi_0 \in H^2_{\!\v{A}}$, then $\psi \in C([0,T_*) , H_{\!\v{A}}^2) \cap C^1 ([0,T_*) , L^2)$.
\item $\| \psi (t) \|_{L^2 (\R^d)} = \| \psi_0 \|_{L^2 (\R^d)}$ and $E_{\mr{S}} [\psi (t) , \v{A}] = E_{\mr{S}} [\psi_0, \v{A}]$. 
\end{enumerate}
\end{thm}

\begin{rem}
Though Theorem \ref{thm:CE88} is proved in the symmetric gauge, by applying a gauge transformation we can obtain a similar Theorem in other gauges. Therefore, the restriction to the symmetric gauge is one of convenience, not necessity. 
\end{rem}

As will be discussed below, Theorem \ref{thm:CE88} also applies in two dimensions. Global existence of $H^1_{\!\v{A}} (\R^d)$-solutions to (\ref{eq:NLMS}) for $p < 1 + 4/d$ and $\mu < 0$ follows from the diamagnetic inequality combined with the same estimates that were applied to (\ref{eq:NLS}) to obtained global existence there.\footnote{For the defocusing case $\mu > 0$, global existence follows for any $1 < p < 1 + 4/(d-2)$ and $\psi_0 \in H^1_{\!\v{A}} (\R^d)$. This follows from the conservation of energy: $\| (\v{p} + \v{A}) \psi (t) \|_{L^2 (\R^d)}^2 \leq E_{\mr{S}} [\psi_0 , \v{A}]$.}  Indeed, the diamagnetic inequality reads
\begin{align}\label{eq:diamagnetic}
| \nabla |\psi| (\v{x}) | \leq | (\v{p} + \v{A}) \psi (\v{x}) | , \hspace{5mm} \text{for a.e.} ~ \v{x} \in \R^d .
\end{align}
From this, the Gagliardo-Nirenberg inequality, and the conservation charge and energy for (\ref{eq:NLMS}), we have the following bound on the kinetic energy:
\begin{align*}
\| (\v{p} + \v{A}) \psi \|_{L^2 (\R^d)}^{2} \leq |E_{\mr{S}} [\psi_0]| + \frac{2C_{p+1}}{p+1} \| (\v{p} + \v{A}) \psi \|_{L^2 (\R^d)}^{d(p-1)/2} .
\end{align*}
Hence, if $p < 1 + 4/d$, then a uniform bound on $\| (\v{p} + \v{A}) \psi \|_{L^2 (\R^d)}$ follows. According to the blow-up alternative of Theorem \ref{thm:CE88}, we have global well-posedness of the Cauchy problem (\ref{eq:NLMS}) for $1 < p < 1 + 4/d$. 

The proof of Theorem \ref{thm:CE88} relies on homogeneous and non-homogeneous Strichartz estimates for the unitary time evolution $U_{\mr{S}} (t) = \exp{ \{- i t (\v{p} + \v{A})^2 \}}$ when $\v{A} = B \v{x}_{\perp} / 2$. To discuss these estimates in detail, we first write
\begin{align}\label{eq:unitary_propagator_schr_constant_field}
U_{\mr{S}} (t) = \left\lbrace \begin{array}{ll}
M(t) , & d = 2 \\
e^{i t \partial_3^2} M (t) , & d = 3 ,
\end{array} \right.
\end{align}
where $M (t)$ is the operator given by
\begin{align}\label{eq:mehler_operator}
M(t) = \exp{ \left\lbrace - i t \left[ \left( p_1 - \frac{B}{2} x_2 \right) + \left( p_2 + \frac{B}{2} x_1 \right) \right] \right\rbrace } .
\end{align}
It is possible to write the integral kernel of the operator $M(t)$ explicitly; see, for example, \cite{AHS78} for a derivation. Known as the \textit{Mehler kernel}, and denoted by the same symbol $M (t) : \R^4 \rightarrow \C$, it reads
\begin{align}\label{eq:mehler_kernel}
M (\v{x} , \v{y}, t) = \frac{B}{4 \pi \sin{(Bt)}} \exp{\left\lbrace \frac{B}{4i} \left( \cot{(Bt)} |\v{x} - \v{y}|^2 - 2 \v{x} \wedge \v{y} \right) \right\rbrace } ,
\end{align}
where $\v{x} \wedge \v{y} = x_1 y_2 - x_2 y_1$. Using the representation (\ref{eq:mehler_kernel}) one has the following $L^p$-estimate on the unitary time evolution for $d = 3$:
\begin{align}\label{eq:US-Lp_estimate}
\| U_{\mr{S}} (t) \psi_0 \|_{L^p (\R^3)} \leq \left( \frac{|B|}{\sqrt{|t|} |\sin{(Bt)}|} \right)^{1 - \frac{2}{p}} \| \psi_0 \|_{L^{p'} (\R^3)} ,
\end{align}
for all $\psi_0 \in L^{p'} (\R^3)$. We note that the time decay in (\ref{eq:US-Lp_estimate}) is due to the free motion in the third direction. Using (\ref{eq:US-Lp_estimate}) it is then possible to show
\begin{align}\label{eq:homogeneous_Strichartz_S_constant_field_3D}
\| U_{\mr{S}} \psi_0 \|_{L^q_t L^r_x ([0,T] \times \R^3)} \leq C_1 \| \psi_0 \|_{L^2 (\R^3)}  , 
\end{align}
and
\begin{align}\label{eq:nonhomogeneous_Strichartz_S_constant_field_3D}
\left\lVert \int_0^t U_{\mr{S}} (t-\tau) \psi (\tau) \dd \tau \right\rVert_{L^q_t L^r_x ([0,T] \times \R^3)} \leq C_2 \| \psi \|_{L^{\tilde{q}'}_t L_x^{\tilde{r}'} ([0,T] \times \R^3)},
\end{align}
for all Schr\"{o}dinger admissible (see (\ref{eq:strichartz_estimates})) $(q,r)$ and $(\tilde{q} , \tilde{r})$ with $d = 3$, and where $C_1 > 0$ depends only on $r$ and $T$, and $C_2 > 0$ depends only on $r$, $\tilde{r}$, and $T$. 

The authors in \cite{cazenave_esteban_88} prove Theorem \ref{thm:CE88} for $d = 3$ using estimates (\ref{eq:homogeneous_Strichartz_S_constant_field_3D}) and (\ref{eq:nonhomogeneous_Strichartz_S_constant_field_3D}). However, using the following Theorem regarding Strichartz estimates for $U_{\mr{S}} (t) \equiv M (t)$ for the two dimensional case, we may easily prove Theorem \ref{thm:CE88} in the $d = 2$ case as well.  Note, the magnetic evolution is periodic with period $\pi/B$ which is essentially the Larmor period. Thus, there is no decay in time and one has to consider the evolution for $0\le t \le \pi/B$. The following may be somewhat surprising.
\begin{thm}\label{thm:strichartz_identity}
For the unitary propagator $U_{\mr{S}} (t) \equiv M (t)$ given by (\ref{eq:mehler_operator})-(\ref{eq:mehler_kernel}) we have the identity
\begin{align}\label{eq:unitary_propagator_identity}
\| M (t) \psi_0 \|_{L_t^q L_x^r ( (0 , \pi / B) \times \R^2 )} = (4 \pi)^{1 - \frac{4}{q}} \| e^{i t \Delta} \psi_0 \|_{L^q_t L^r_x (\R \times \R^2)} ,
\end{align}
for all $\psi_0 \in L^2 (\R^2)$ and all Schr\"{o}dinger admissible exponents $(q,r)$.
\end{thm}
\begin{proof}
Observe that we may write the Mehler kernel (\ref{eq:mehler_kernel}) as
\begin{align*}
M (\v{x} , \v{y}) = \frac{B}{4\pi \sin{(Bt)}} \exp{ \left\lbrace \frac{B}{4i} \cot{(Bt)} \left( |\v{x}|^2 + |\v{y}|^2 \right) \right\rbrace } \exp{ \left\lbrace i \frac{B \v{y} \cdot R(Bt) \v{x}}{2 \sin{(Bt)}} \right\rbrace } ,
\end{align*}
where $R(\theta)$ is the usual $2 \times 2$ rotation matrix given by
\begin{align*}
R (\theta) = \left( \begin{array}{cc}
\cos{(\theta)} & - \sin{(\theta)} \\
\sin{(\theta)} & \cos{(\theta)} 
\end{array} \right) .
\end{align*}  The unitary propagator $U_{\mr{S}} (t) \equiv M(t)$ acting on $\psi_0 : \R^2 \rightarrow \C$ may be written
\begin{align*}
( M (t) \psi_0 ) (\v{x}) = \frac{B \exp{ \left\lbrace \frac{B}{4i} \cot{(Bt)} |\v{x}|^2 \right\rbrace } }{4\pi \sin{(Bt)}} \int_{\R^2} \exp{ \left\lbrace  i \frac{B \v{y} \cdot R(Bt) \v{x}}{2 \sin{(Bt)}} \right\rbrace } g(\v{y},t) \dd \v{y} ,
\end{align*}
where
\begin{align*}
g(\v{y},t) = \exp{ \left\lbrace \frac{B}{4i} \cot{(Bt)} |\v{y}|^2 \right\rbrace } \psi_0 (\v{y}) .
\end{align*}
Therefore,
\begin{align*}
( M (t) \psi_0 ) (\v{x}) = \frac{B\exp{ \left\lbrace \frac{B}{4i} \cot{(Bt)} |\v{x}|^2 \right\rbrace }}{4\pi \sin(Bt)}  \cF g \left( - \frac{B R(Bt) \v{x}}{4\pi \sin(Bt)} , t \right) ,
\end{align*}
where $\cF$ is the Fourier transform defined by
\begin{align*}
( \cF f )(\v{k}) = \int_{\R^d} e^{- 2 \pi i \v{k} \cdot \v{x}} f(\v{x}) \dd \v{x} . 
\end{align*}
Letting $r \geq 2$, we may now compute
\begin{align*}
\int_{\R^2} |  ( M (t) \psi_0 ) (\v{x}) |^r \dd \v{x} & = \left| \frac{B}{4\pi \sin(Bt)} \right|^r \int_{\R^2} \left| \cF g \left( \frac{B}{4\pi \sin(Bt)} \v{x} , t \right) \right|^r \dd \v{x} \\
& = \left| \frac{B}{4\pi \sin(Bt)}\right|^{r-2} \int_{\R^2} |\cF g (\v{x}, t)|^r d \v{x} .
\end{align*}
If we raise this to the $q/r$-power, integrate over $t$ from $0$ to $\pi / B$, and use the substitution $s = B \cot{(Bt)}$, we find
\begin{align*}
& \int_0^{\pi/B} \left( \int_{\R^2} | ( M(t) \psi_0 ) (\v{x}) |^r \dd \v{x} \right)^{\frac{q}{r}} \dd t \\
& = \frac{1}{(4 \pi)^2} \int_{-\infty}^{\infty} \left| \frac{B}{4\pi \sin(Bt)}\right|^{\frac{q}{r} (r - 2) - 2} \left( \int_{\R^2} \left| \cF \left( e^{- i s | \cdot |^2 / 4} \psi_0 \right) (\v{x}) \right|^r \dd \v{x} \right)^{\frac{q}{r}} \dd s .
\end{align*}
If $(q,r)$ are Schr\"{o}dinger admissible, we have
\begin{align*}
\frac{q}{r} (r - 2) - 2 = 0 .
\end{align*}
This last observation together with the previous calculation yields (\ref{eq:unitary_propagator_identity}).
\end{proof}
Theorem \ref{thm:strichartz_identity} states in essence that the Strichartz estimates for the Mehler kernel (\ref{eq:mehler_kernel}) are the same as those for the free Schr\"{o}dinger evolution. In particular, we bring attention to an interesting corollary to Theorem \ref{thm:strichartz_identity}.
\begin{cor}
For $(q,r) = (4,4)$, the identity (\ref{eq:unitary_propagator_identity}) implies the sharp constant for the Strichartz estimate for the Mehler kernel (\ref{eq:mehler_kernel}) is the same as for the free Schr\"{o}dinger evolution, namely $1 / \sqrt{2}$ \cite{HD06}.
\end{cor}
We refer the reader to \cite{HD06} for a derivation of the sharp constant for the free Schr\"odinger Strichartz estimate in the $(q,r,d) = (4,4,2)$ case. The proof of the previous corollary follows directly from the formula (\ref{eq:unitary_propagator_identity}). Lastly, we mention that a simple formula like (\ref{eq:unitary_propagator_identity}) doesn't seem possible in $3$D.

\subsection{Finite Time Blow-up for Uniform Magnetic Fields}\label{sub:blowup}

We introduce the function space
\begin{align*}
\Sigma_{\mr{S}} = \{ f \in H^1_{\!\v{A}} (\R^d ; \C) ~ : ~ \int_{\R^d} |\v{x}|^2 |f (\v{x})|^2 \dd \v{x} < \infty \} ,
\end{align*}
and define $F_{\mr{S}} : \Sigma_{\mr{S}} \rightarrow \R$ by 
\begin{align}\label{def:blowup_functional_schr}
F_{\mr{S}} [\psi , \v{A}] = E_{\mr{S}} [\psi , \v{A}] - B \re{ \langle \v{x}_{\perp} \psi , (\v{p} + \v{A}) \psi \rangle_{L^2(\R^d)} } + \frac{B^2}{2} \| \rho \psi \|_{L^2 (\R^d)}^2 ,
\end{align}
where $\rho = \sqrt{x_1^2 + x_2^2}$.\footnote{We will frequently switch between the polar/cylindrical coordinates $(\rho , \vartheta)$ and $(\rho , \vartheta , x_3)$ and the Cartesian coordinates $(x_1 , x_2)$ and $(x_1 , x_2 , x_3)$ for $\R^2$ and $\R^3$, respectively. Here $\rho = \sqrt{ x_1^2 + x_2^2 }$ and $\tan{\vartheta} = x_2 / x_1$.} The functional $F_{\mr{S}} [\psi , \v{A}]$ defined on $\Sigma_{\mr{S}}$ will play a key role in providing a sufficient condition for blow-up of solutions to (\ref{eq:NLMS}) when $p \geq 1 + 4/d$ and $\mu < 0$. Since the vector potential $\v{A}$ will be fixed we will usually suppress the $\v{A}$-dependence of $F_{\mr{S}}$ and simply write $F_{\mr{S}} [\psi]$. 

Observe that the functional $F_{\mr{S}}$ is gauge invariant, which is the main reason why we didn't choose to fix any particular gauge from the start. Indeed, if we select the symmetric gauge $\v{A} = B \v{x}_{\perp} / 2$, then the term
\begin{align*}
- B \re{ \langle \v{x}_{\perp} \psi , (\v{p} + \v{A}) \psi \rangle_{L^2 (\R^d)} } + \frac{B^2}{2} \| \rho \psi \|_{L^2 (\R^d)}^2
\end{align*} 
in (\ref{def:blowup_functional_schr}) simply equals $- B \langle \psi , L_3 \psi \rangle_{L^2 (\R^d)}$, where $L_3 \equiv \v{x}_{\perp} \cdot \v{p} = - i \partial_{\vartheta} = - i ( - x_2 \partial_1 + x_1 \partial_2 )$ is the $x_3$-component of the angular momentum. Since we are considering a uniform magnetic field along the $x_3$-axis, it is reasonable to assert the $x_3$-component of the angular momentum is preserved. Indeed, at least formally, 
\begin{align*}
\frac{\dd}{\dd t} \langle \psi , L_3 \psi \rangle_{L^2 (\R^d)} = \int_{\R^d} |\psi|^2 \partial_{\vartheta} |\psi|^{p-1} = \frac{p-1}{p+1} \int_{\R^d} \partial_{\vartheta} |\psi|^{p+1} = 0 . 
\end{align*}
The next Lemma makes this precise, showing that, in any gauge, $F_{\mr{S}}$ is conserved under the time evolution of (\ref{eq:NLMS}).
\begin{lem}\label{lem:conservation_of_Fs}
Let $d \in \{2,3\}$ and $\v{A} \in L_{\mr{loc}}^2 (\R^d ; \R^d)$ generate a uniform magnetic field $\v{B} = (0,0,B)$. Let $\psi_0 \in H_{\!\v{A}}^1 (\R^d)$ and $\psi \in C([0,T_*) , H_{\!\v{A}}^1) \cap C^1 ([0,T_*) , H_{\!\v{A}}^{-1})$ denote the corresponding solution to (\ref{eq:NLMS}). Then, $F_{\mr{S}} [\psi (t)] = F_{\mr{S}} [\psi_0]$. 
\end{lem}

The proof of Lemma \ref{lem:conservation_of_Fs} is given at the end of \S\ref{sub:virial}. Using that $F_{\mr{S}} [\psi]$ is conserved we have the following Theorem concerning the second time derivative of the expectation value of $|\v{x}|^2$.

\begin{thm}\label{thm:blowup_NLMS}
Let $d \in \{ 2, 3 \}$, $1 < p < 1 + 4/(d-2)$, $\v{A} \in L_{\mr{loc}}^2 (\R^d ; \R^d)$ generate a uniform magnetic field $\v{B} = (0,0,B)$, and $\psi_0 \in \Sigma_{\mr{S}}$. Let $\psi \in C ( [0,T_*) ; \Sigma_{\mr{S}} ) \cap C^1 ( [0 , T_*) ; H_{\!\v{A}}^{-1} )$ be the corresponding maximal solution to the initial value problem (\ref{eq:NLMS}). Then, the function $g (t) = \frac{1}{4} \| \v{x} \psi (t) \|_{L^2 (\R^d)}^2$ satisfies the virial identity:
\begin{align}\label{eq:constantfield_xsquare_gauge_independent}
\ddot{g} (t) = 2 F_{\mr{S}} [\psi_0] + \mu d \frac{p - (1 + 4 / d)}{p+1} \| \psi (t) \|_{L^{p+1} (\R^d)}^{p+1} - B^2 \| \rho \psi (t) \|_{L^2 (\R^d)}^2 .
\end{align}
In particular, if $d=2$ and $p = 3$, then (\ref{eq:constantfield_xsquare_gauge_independent}) becomes a second-order equation for $g$ that can be solved exactly:
\begin{align}\label{eq:constantfield_explicitsolution}
g(t) = \frac{F_{\mr{S}} [\psi_0]}{2 B^2} + \left( g_0 - \frac{F_{\mr{S}} [\psi_0]}{2 B^2} \right) \cos{(2Bt)} + \frac{\dot{g}_0}{2 B} \sin{(2Bt)}
\end{align}
where $g_0 = \frac{1}{4} \| \rho \psi_0 \|_{L^2 (\R^2)}^2$ and $\dot{g}_0 = \re{ \langle \v{x} \psi_0 , (\v{p} + \v{A}) \psi_0 \rangle_{L^2 (\R^2)}}$.
\end{thm}

\begin{rem}
It is not obvious that, for $\psi_0 \in \Sigma_{\mr{S}}$, the corresponding solution is in $\psi \in C ( [0,T_*) ; \Sigma_{\mr{S}} ) \cap C^1 ( [0 , T_*) ; H_{\!\v{A}}^{-1} )$. This fact is shown in \cite{ribeiro91} in the course of proving of Theorem 1.2 there. 
\end{rem}

The proof of Theorem \ref{thm:blowup_NLMS} is reserved for \S\ref{sub:virial}. Following a similar reasoning as discussed in \S\ref{sec:intro} for (\ref{eq:NLS}), observe that if $\mu < 0$ and $p \geq 1 + 4/d$, then $\ddot{g}(t) \leq 2 F_{\mr{S}} [\psi_0]$. So if $F_{\mr{S}} [\psi_0] < 0$, then $\ddot{g} (t) < 0$ for all times $t \in [0 , T_*)$. If $T_* = + \infty$, then $g$ would necessarily hit $0$ at some $t_* > 0$, implying (by the uncertainty principle) that $\| (\v{p} + \v{A}) \psi (t_*) \|_{L^2 (\R^d)} = + \infty$. This contradicts the blow-up alternative of Theorem \ref{thm:CE88}.
\begin{cor}\label{cor:blow_up_schr}
Suppose $\mu < 0$, $1 + 4/d \leq p < 1 + 4/(d-2)$, $\psi_0 \in \Sigma_{\mr{S}}$, and $F_{\mr{S}}[\psi_0] < 0$. Then the corresponding solution to (\ref{eq:NLMS}) blows up in finite time. If $F_{\mr{S}} [\psi_0] = 0$, then blow-up occurs when $\dot{g} (0) = \re{ \langle \v{x} \psi_0 , (\v{p} + \v{A}) \psi_0 \rangle_{L^2 (\R^d)}} < 0$. 
\end{cor}

Note that (\ref{eq:constantfield_explicitsolution}) implies the time for blow-up to occur decreases as the magnetic field strength increases.

Consider the explicit solution (\ref{eq:constantfield_explicitsolution}) for $d = 2$ and $p = 3$ from Theorem \ref{thm:blowup_NLMS}. The condition for $g(t) < 0$ at some time $t \in (0 , T_*)$ is given by
\begin{align*}
\left( g_0 - \frac{F_{\mr{S}}[\psi_0]}{2 B^2} \right)^2 + \left( \frac{\dot{g}_0}{2B} \right)^2 > \left( \frac{F_{\mr{S}}[\psi_0]}{2 B^2} \right)^2 .
\end{align*}
After some algebra this simplifies to
\begin{align}\label{eq:condition_for_blowup_p=3_d=2_(1)}
F_{\mr{S}}[\psi_0] g_0 < B^2 \left( g_0^2 + \frac{\dot{g}_0^2}{4 B^2} \right) .
\end{align}
Observe that (\ref{eq:condition_for_blowup_p=3_d=2_(1)}) is much weaker than demanding $F_{\mr{S}}[\psi_0] < 0$ for blow-up to occur. Consider the special case $\dot{g}_0 = 0$. Using the definition (\ref{def:blowup_functional_schr}) of $F_{\mr{S}}[\psi_0]$ and choosing the symmetric gauge $\v{A} = \frac{B}{2} \v{x}_{\perp}$, (\ref{eq:condition_for_blowup_p=3_d=2_(1)}) then reduces to $E_{\mr{S}} [\psi_0 , \v{0}] < 0$. This is clearly only satisfied in the focusing case $\mu < 0$. 

In the general case $\dot{g}_0 \neq 0$, (\ref{eq:condition_for_blowup_p=3_d=2_(1)}) reduces to
\begin{align}\label{eq:condition_for_blowup_p=3_d=2_(2)}
F_{\mr{S}} [\psi_0] - \frac{B^2}{4} \| \rho \psi_0 \|_{L^2 (\R^2)}^2 < \frac{\re{ \langle \v{x} \psi_0 , (\v{p} + \v{A}) \psi_0 \rangle^2_{L^2 (\R^2)} }}{\| \rho \psi_0 \|_{L^2 (\R^2)}^2} .
\end{align}
Choosing again the symmetric gauge $\v{A} = \frac{B}{2} \v{x}_{\perp}$ and using the definition (\ref{def:NLMS_energy}) of $E_{\mr{S}}$, then the expression (\ref{eq:condition_for_blowup_p=3_d=2_(2)}) further reduces to 
\begin{align}\label{eq:condition_for_blowup_p=3_d=2_symmetric}
E_{\mr{S}} [\psi_0 , \v{0}] < \frac{\re{ \langle \v{x} \psi_0 , \v{p}  \psi_0 \rangle^2_{L^2 (\R^2)}}}{\| \rho \psi_0 \|_{L^2 (\R^2)}^2} .
\end{align}
Since $\re{ \langle \v{x} \psi_0 , \v{p} \psi_0 \rangle_{L^2 (\R^2)}} \leq \| \rho \psi_0 \|_{L^2 (\R^2)} \| \nabla \psi_0 \|_{L^2 (\R^2)}$ we see that the inequality (\ref{eq:condition_for_blowup_p=3_d=2_symmetric}) is only satisfied if $\mu < 0$. We note these observations are a consistency check with the earlier observation that the defocusing $H^1_{\!\v{A}} (\R^2)$-subcritical NLMS equation is globally well-posed.  

On a final note for this section, the sufficient condition for blow-up, namely $F_{\mr{S}} [\psi_0] < 0$, is significantly different than the one found in \cite{ribeiro91, garcia2011}. There, the authors claim, based on a virial identity argument, that $E_{\mr{S}} [\psi_0 , \v{A}] < 0$ is sufficient for blow-up to occur. For this reason we spend the next several paragraphs closely analyzing the relationship between $F_{\mr{S}} [\psi_0]$ and $E_{\mr{S}} [\psi_0]$. Choose the symmetric gauge $\v{A} = B \v{x}_{\perp} / 2$, $B > 0$, and fix $\psi_0 \in \Sigma_{\mr{S}}$. By expanding the kinetic energy $T_{\mr{S}} [\psi_0 , \v{A}] = \| (\v{p} + \v{A}) \psi_0 \|_{L^2 (\R^d)}^2$ we have that
\begin{align}\label{eq:FS_expanded}
F_{\mr{S}} [\psi_0 , \v{A}] = E_0 + \frac{B^2}{4} \| \rho \psi_0 \|_{L^2 (\R^d)}^2 ,
\end{align}
and
\begin{align}\label{eq:ES_expanded}
E_{\mr{S}} [\psi_0 , \v{A}] = E_0 - B \langle L_3 \rangle_0 + \frac{B^2}{4} \| \rho \psi_0 \|_{L^2 (\R^d)}^2 ,
\end{align}
where $E_0 \equiv E_{\mr{S}} [\psi_0 , \v{0}] = \| \nabla \psi_0 \|_{L^2 (\R^d)}^2 - \frac{2}{p+1} \| \psi_0 \|_{L^{p+1}(\R^d)}^{p+1}$ and $\langle L_3 \rangle_{0} \equiv \langle L_3 \psi_0 , \psi_0 \rangle_{L^2 (\R^d)}$. Clearly, for $B = 0$, $F_{\mr{S}} [\psi_0 , \v{0}] = E_0$. From the expressions (\ref{eq:FS_expanded}) and (\ref{eq:ES_expanded}) we can immediately conclude
\begin{align*}
\left\lbrace \begin{array}{lc}
F_{\mr{S}} [\psi_0 , \v{A}] < E_{\mr{S}} [\psi_0 , \v{A}] , & \text{when} ~ \langle L_3 \rangle_{0} < 0, \\
F_{\mr{S}} [\psi_0 , \v{A}] > E_{\mr{S}} [\psi_0 , \v{A}] , & \text{when} ~ \langle L_3 \rangle_{0} > 0, \\
F_{\mr{S}} [\psi_0 , \v{A}] = E_{\mr{S}} [\psi_0 , \v{A}] , & \text{when} ~ \langle L_3 \rangle_{0} = 0.
\end{array} \right.
\end{align*}

The previous inequalities suggest that, when $\langle L_3 \rangle_{0} < 0$, it is possible to derive a general criterion on the magnetic field strength $B > 0$ and the initial data $\psi_0 \in \Sigma_{\mr{S}}$ that gives $E_{\mr{S}} [\psi_0 , \v{A}] > 0$ and $F_{\mr{S}} [\psi_0 , \v{A}] < 0$, and vice versa when $\langle L_3 \rangle_0 > 0$. We first note that it is clear from (\ref{eq:FS_expanded}) that we must assume $E_0 < 0$ for $F_{\mr{S}} [\psi_0 , \v{A}] < 0$ to be possible. Assuming $E_0 < 0$, from (\ref{eq:FS_expanded}) we see that if $B^2 < 4 |E_0| / \| \rho \psi_0 \|_{L^2 (\R^d)}^2$, then $F_{\mr{S}} [\psi_0 , \v{A}] < 0$, when $E_{\mr{S}} [\psi_0 , \v{0}] < 0$. Furthermore, assuming $\langle L_3 \rangle_{0} < 0$, from (\ref{eq:ES_expanded}) we observe that to have $E_{\mr{S}} [\psi_0 , \v{A}] > 0$ we must have
\begin{align*}
B |\langle L_3 \rangle_{0}| + \frac{B^2}{4} \| \rho \psi_0 \|_{L^2 (\R^d)}^2 > | E_0 |  .
\end{align*}
Using this estimate to ensure $E_{\mr{S}} [\psi_0 , \v{A}] > 0$ we choose $B > |E_0 / \langle L_3 \rangle_{0}|$. Therefore, for $E_{\mr{S}} [\psi_0 , \v{A}] > 0$ and $F_{\mr{S}} [\psi_0 , \v{A}] < 0$, we may choose $B > 0$ such that
\begin{align*}
\left| \frac{E_0}{\langle L_3 \rangle_{0}} \right| < B < \frac{2 \sqrt{|E_0|}}{  \| \rho \psi_0 \|_{L^2 (\R^d)} } .
\end{align*}
Thus, for such a $B > 0$ to exist, it will be necessary to have 
\begin{align*}
 \sqrt{ \left| E_0 \right| } \| \rho \psi_0 \|_{L^2 (\R^d)} < 2 |\langle L_3 \rangle_{0}| . 
\end{align*}
Such an inequality can certainly be satisfied. Indeed, take $d = 2$, $p = 5$, $\psi_0 (\rho , \vartheta) = u(\rho) e^{- i \vartheta}$ as an $L_3$-eigenstate with eigenvalue $- 1$ with 
\begin{align*}
u (\rho) = \frac{800 \rho}{\sqrt{\pi}} e^{- 400 \rho^2} . 
\end{align*}
Then, for this state, $E_0 = 1600(1 - \frac{800}{81 \pi^2} )$ and $| \langle L_3 \rangle_{0} |^2 / \| \rho \psi_0 \|_{L^2 (\R^d)}^2 = 800 \pi$, and for any $2 < B < 106$ will produce a positive $E_{\mr{S}} [\psi_0 , \v{A}]$, but negative $F_{\mr{S}} [\psi_0 , \v{A}]$. 

\subsection{Virial Identity and Proof of Main Result}\label{sub:virial}

The virial identities in this section for the NLMS equation (\ref{eq:NLMS}) are already present in the literature. The linear case in any space dimension $d \geq 2$ is covered in \cite[Theorem 1.2]{fanelli_vega_09}, while the non-linear generalization can be found in \cite[Theorem 3.1]{garcia2011}. We rederive these identities for completeness, as well as express them in a form that will be useful for the proof of Theorem \ref{thm:blowup_NLMS}. We treat the case of any dimension $d \geq 2$ and a general, time-independent, external magnetic field (i.e., not necessarily a uniform field). The vector potential $\v{A} : \R^d \rightarrow \R^d$ generates the matrix-valued magnetic field $\v{B} : \R^d \rightarrow M_{n \times n} (\R)$ with components $B_{ij} = \partial_j A_i - \partial_i A_j$. We record this first virial identity as the following Lemma. \begin{lem}[\cite{fanelli_vega_09, garcia2011}]\label{lem:NLMS_Virial}
Let $g (t) = \frac{1}{4} \| \v{x} \psi (t) \|_{L^2 (\R^d)}^2$. Then, for any solution $\psi$ to the NLMS equation (\ref{eq:NLMS}) with initial data $\psi_0 \in H_{\!\v{A}}^2 (\R^d ; \C)$, the following virial identity holds:
\begin{align}\label{eq:xsquared_general_dimension}
\ddot{g} = 2 T_{\mr{S}} [\psi , \v{A}] + \mu d \frac{p-1}{p+1} \| \psi \|_{L^{p+1} (\R^d)}^{p+1} - 2 \re{ \langle \v{B} \v{x} \psi , (\v{p} + \v{A}) \psi \rangle_{L^2 (\R^d)} } . 
\end{align}
\end{lem}
\begin{proof}
By Theorem \ref{thm:CE88} the corresponding solution $\psi \in C([0,T_*) , H_{\!\v{A}}^2) \cap C^1 ([0,T_*) , L^2)$ and therefore all the following computations are justified. Consider the function $f(t) = \langle \psi (t) , G (\v{x}) \psi (t) \rangle$, where $G : \R^d \rightarrow \R$ is a differentiable, radial multiplier to be specified later. We denote $H_{\mr{S}} = \gvpi^2 + \mu |\psi|^{p-1}$ where $\gvpi = \v{p} + \v{A}$. Using the identity $[A^2 , B] = A [A,B] + [A,B] A$ and taking the time derivative of $f$ we find 
\begin{align*}
\dot{f} & =\langle i [ \gvpi^2 , G ] \psi , \psi \rangle_{L^2(\R^d)} \\
& =  \langle i ( \gvpi \cdot [ \v{p} , G ] + [ \v{p} , G ] \cdot \gvpi ) \psi , \psi \rangle_{L^2(\R^d)} \\
& = 2 \re{ \langle \nabla G \cdot \gvpi \psi , \psi \rangle_{L^2(\R^d)}} .
\end{align*} 
Choosing $G (\v{x}) = |\v{x}|^2$ we arrive at
\begin{align}
\dot{g} = \re{ \langle \v{x} \cdot \gvpi \psi , \psi \rangle_{L^2(\R^d)}} = \langle ( D + \v{x} \cdot \v{A}) \psi , \psi \rangle_{L^2(\R^d)},
\end{align}
where $D = (\v{x} \cdot \v{p} + \v{p} \cdot \v{x} ) / 2$ is the dilation operator.

For the second time derivative we find
\begin{align}\label{eq:varia_1}
\ddot{g} & = \langle i [ H_{\mr{S}} , (D + \v{x} \cdot \v{A}) ] \psi , \psi \rangle_{L^2(\R^d)}.
\end{align}
Recall that
\begin{align*}
\frac{\dd }{\dd \theta} e^{- i \theta D} H_{\mr{S}} e^{i \theta D} \Big|_{\theta = 0} = i [ H_{\mr{S}} , D ] ,
\end{align*}
and that $(e^{i \theta D} f)(\v{x}) = e^{d \theta /2} f(e^{\theta} \v{x})$. Our task is to compute $e^{-i \theta D} H_{\mr{S}} e^{i \theta D}$. This is straightforward. For example, for a suitable $f$, we have
\begin{align*}
e^{-i \theta D} \gvpi e^{i \theta D} f (\v{x}) = ( (e^{\theta} \v{p} + \v{A} (e^{-\theta} \cdot)) f ) (\v{x}) ,
\end{align*}
which implies $e^{-i \theta D} \gvpi^2 e^{i \theta D} = e^{2 \theta} \gvpi_{\theta}^2$, where $\v{A}_{\theta} (\v{x}) = e^{-\theta} \v{A} (e^{-\theta} \v{x})$ and $\gvpi_{\theta} = \v{p} + \v{A}_{\theta}$. Similar computations yield
\begin{align*}
e^{-i \theta D} H_{\mr{S}} e^{i \theta D} & = e^{2 \theta} \gvpi_{\theta}^2 + \mu |\psi (e^{- \theta} \cdot )|^{p-1} .
\end{align*} 
Differentiating the previous expression with respect to $\theta$ and evaluating it at $\theta = 0$ we find
\begin{align}\label{eq:varia_2}
i [H_{\mr{S}} , D] & = 2 \gvpi^2 - \mu (\v{x} \cdot \nabla ) |\psi|^{p-1} + \frac{\dd}{\dd \theta} \gvpi_{\theta}^2 \Big|_{\theta = 0} .
\end{align}

To complete the computation of (\ref{eq:varia_1}) we must workout the commutator $[H_{\mr{S}} , \v{x} \cdot \v{A}]$, which reduces to $ [\gvpi^2 , \v{x} \cdot \v{A}]$. One finds
\begin{align*}
[\gvpi^2 , \v{x} \cdot \v{A}] = \gvpi \cdot [\v{p} , \v{x} \cdot \v{A}] + [\v{p} , \v{x} \cdot \v{A}] \cdot \gvpi ,
\end{align*}
and
\begin{align*}
[\nabla , \v{x} \cdot \v{A}] & = \nabla (\v{x} \cdot \v{A}) = \v{A} + (\v{x} \cdot \nabla) \v{A} - \sum_{ij} x_j B_{ij} \v{e}_i .
\end{align*}
Since
\begin{align*}
\v{A} (\v{x}) + (\v{x} \cdot \nabla) \v{A} (\v{x}) = - \frac{\dd}{\dd \theta} \Big|_{\theta = 0} \v{A}_{\theta} (\v{x}) ,
\end{align*}
we conclude that
\begin{align}\label{eq:varia_3}
i [ H_{\mr{S}} , \v{x} \cdot \v{A}]  = - \frac{\dd}{\dd \theta} \Big|_{\theta = 0} \gvpi_{\theta}^2 - \left( \gvpi \cdot \v{B} \v{x} + \v{B} \v{x} \cdot \gvpi \right) .
\end{align}
Combining (\ref{eq:varia_1}), (\ref{eq:varia_2}), and (\ref{eq:varia_3}) we conclude (\ref{eq:xsquared_general_dimension}).  
\end{proof}

\begin{proof}[Proof of Lemma \ref{lem:conservation_of_Fs} and Theorem \ref{thm:blowup_NLMS}]
We specialize to the case $d = 3$, as $d = 2$ is similar. We will consider $\psi_0 \in H^2_{\!\v{A}} (\R^3) \cap \Sigma_{\mr{S}}$, as the general case of $\psi_0 \in \Sigma_{\mr{S}}$ will follow from the continuous dependence portion of Theorem \ref{thm:CE88}. Again, we denote $\gvpi = \v{p} + \v{A}$. Begin by noting that, in dimensions $d \in \{2,3\}$, with $\v{B} = (0,0,B)$ a uniform field, (\ref{eq:xsquared_general_dimension}) becomes
\begin{align}\label{eq:xsquared_3_dimension}
\ddot{g} = 2 T_{\mr{S}} [\psi , \v{A}] + \mu d \frac{p-1}{p+1} \| \psi \|_{L^{p+1} (\R^d)}^{p+1} - 2 \re{ \langle B \v{x}_{\perp} \psi , \gvpi \psi \rangle_{L^2 (\R^d)} } .
\end{align} 
Therefore, the proof of Lemma \ref{lem:conservation_of_Fs} boils down to showing 
\begin{align*}
- B \re{ \langle \v{x}_{\perp} \psi , \gvpi \psi \rangle_{L^2 (\R^3)} } + \frac{B^2}{2} \| \rho \psi \|_{L^2 (\R^3)}^2
\end{align*}
is conserved. We start by computing the time derivative of $\langle \psi , \v{x}_{\perp} \cdot \gvpi \psi \rangle_{L^2(\R^d)}$. We first note
\begin{align}\label{eq:proof_thmNLMS_2}
\frac{\dd }{\dd t} \langle \psi , \v{x}_{\perp} \cdot \gvpi \psi \rangle_{L^2(\R^3)} = \langle i [ \gvpi^2 , \v{x}_{\perp} \cdot \gvpi ] \psi , \psi \rangle_{L^2(\R^3)} + \mu \langle i (L_3 |\psi|^{p-1}) \psi , \psi \rangle_{L^2 (\R^3)} ,
\end{align}
where we recall that $L_3 \equiv \v{x}_{\perp} \cdot \v{p} = - i \partial_{\vartheta}$. The second commutator on the right hand side of (\ref{eq:proof_thmNLMS_2}) is straightforward to compute:
\begin{align*}
\langle i (L_3 |\psi|^{p-1}) \psi , \psi \rangle_{L^2 (\R^3)} = \frac{p-1}{p+1} \int_{\R^3} \partial_{\vartheta} |\psi|^{p+1} = 0 . 
\end{align*}

To compute the first commutator on the right hand side of (\ref{eq:proof_thmNLMS_2}) we note 
\begin{align*}
[ \gvpi^2 , \v{x}_{\perp} \cdot \gvpi ] = \sum_{j,k=1}^3 \left( \pi_j [ \pi_j , (\v{x}_{\perp})_k \pi_k]  + [ \pi_j , (\v{x}_{\perp})_k \pi_k] \pi_j \right) . 
\end{align*}
Noting that $[\pi_3 , \pi_k] = - i ( \partial_3 A_k - \partial_k A_3 ) = 0$ in a uniform magnetic field directed along the $x_3$-axis, the above sum reduces to
\begin{align*}
[ \gvpi^2 , \v{x}_{\perp} \cdot \gvpi ] & = \sum_{j,k=1}^2 \left( \pi_j [ \pi_j , (\v{x}_{\perp})_k \pi_k] + [ \pi_j , (\v{x}_{\perp})_k \pi_k] \pi_j \right) \\
& = \pi_1 [ \pi_1 , x_1 \pi_2 ] - \pi_2 [ \pi_2 , x_2 \pi_1 ] . \numberthis \label{eq:proof_thmNLMS_3} 
\end{align*}
Since $\partial_1 A_2 - \partial_2 A_1 = B$, for the first commutator in (\ref{eq:proof_thmNLMS_3}) we find
\begin{align*}
i [ \pi_1 , x_1 \pi_2 ] =  \pi_2 + i x_1 [\pi_1 , \pi_2] = \pi_2 + B x_1 , 
\end{align*}
and for the second commutator in (\ref{eq:proof_thmNLMS_3}) we find
\begin{align*}
i [ \pi_2 , x_2 \pi_1 ] =  \pi_1 + i x_2 [\pi_2 , \pi_1] = \pi_1 - B x_2 .
\end{align*}
Plugging the previous two commutators back into (\ref{eq:proof_thmNLMS_3}) we conclude
\begin{align*}
i [ \gvpi^2 , \v{x}_{\perp} \cdot \gvpi ] & = \pi_1 \left( \pi_2 + B x_1  \right) - \pi_2 \left( (p_1 + A_2) - B x_2 \right) \\
& = - 2 i B + 2  B (x_1 , x_2 , 0) \cdot \v{A} + 2 B (x_1 , x_2 , 0) \cdot \v{p} .
\end{align*}
Therefore, (\ref{eq:proof_thmNLMS_2}) becomes
\begin{align}\label{eq:proof_thmNLMS_4}
\frac{\dd }{\dd t} \re{ \langle \v{x}_{\perp} \psi , \gvpi \psi \rangle_{L^2(\R^3)}} = 2 B \re{ \langle (x_1 , x_2 , 0) \cdot ( \v{p} + \v{A} ) \psi , \psi \rangle_{L^2(\R^3)}} .
\end{align}
It is easily verified that the right hand side of (\ref{eq:proof_thmNLMS_4}) is proportional to the time derivative of $\| \rho \psi \|_{L^2 (\R^3)}^2$, where $\rho^2 = x_1^2 + x_2^2$. That is, we have the desired identity
\begin{align}\label{eq:proof_thmNLMS_5} 
\frac{\dd }{\dd t} \re{ \langle \v{x}_{\perp} \psi , \gvpi \psi \rangle_{L^2(\R^3)}} = \frac{B}{2} \frac{\dd}{\dd t} \| \rho \psi \|_{L^2 (\R^3)}^2 .  
\end{align}
Finishing the proof of Theorem \ref{thm:blowup_NLMS} simply amounts to rewriting the identity (\ref{eq:xsquared_3_dimension}) and using (\ref{eq:proof_thmNLMS_5}).
\end{proof}

\section{The Non-linear Pauli Equation}\label{sec:NLP}

In this final section, we consider generalizing earlier results on the NLMS equation to the non-linear Pauli (NLP) equation. In space dimensions $d \in \{2,3\}$, the NLP equation\footnote{Note that for $d = 2$, the NLP equation is not equivalent to the NLMS equation via a gauge transformation. This is a consequence of the non-linearity "mixing" the components of $\psi$.} reads
\begin{align}\label{eq:NLP}
\left\lbrace  \begin{array}{l}
i \partial_t \psi = [ \gvsig \cdot (\v{p} + \v{A})]^2 \psi + \mu | \psi|^{p-1} \psi \\
\psi (0 , \v{x}) = \psi_0 (\v{x}) ,
\end{array} \right. 
\end{align}
where $\psi : \R^d \rightarrow \C^2$, and
\begin{align*}
\gvsig = \left\lbrace \begin{array}{cc}
(\sigma_1 , \sigma_2 , \sigma_3) , & d = 3 , \\
(\sigma_1 , \sigma_2) , & d = 2 ,
\end{array}  \right. 
\end{align*}
is the vector of Pauli matrices, which are $2 \times 2$ Hermitian matrices assumed to satisfy the commutation relations $[\sigma^j , \sigma^k] = 2 i \epsilon_{jk\ell} \sigma^{\ell}$ and anticommutation relations $\left\lbrace \sigma^j , \sigma^k \right\rbrace = 2 \delta_{jk} I$. A typical representation is
\begin{align*}
\sigma_1 = \left( \begin{array}{cc}
0 & 1 \\
1 & 0 
\end{array} \right) , \hspace{1cm} \sigma_2 = \left( \begin{array}{cc}
0 & -i \\
i & 0 
\end{array} \right) , \hspace{1cm} \sigma_3 = \left( \begin{array}{cc}
1 & 0 \\
0 & - 1
\end{array} \right) .
\end{align*}
We typically consider (\ref{eq:NLP}) as an initial value problem in the space $H^1_{\v{A}} (\R^d ; \C^2)$, which is the obvious generalization of the space $H_{\v{A}}^1 (\R^d ; \C)$ discussed in \S\ref{sec:NLMS}. Again, we will always assume $\v{A} \in L^2_{\mr{loc}} (\R^d ; \R^d)$ is such that $[\gvsig \cdot (\v{p} + \v{A})]^2$ is an essentially self-adjoint operator on $L^2 (\R^d ; \C^2)$ with domain $H_{\!\v{A}}^2 (\R^d ; \C^2)$. The total energy of (\ref{eq:NLP}) is
\begin{align}\label{def:total_mag_energy_pauli}
E_{\mr{P}} [\psi , \v{A}] (t) =  T_{\mr{P}} [\psi, \v{A}] + \frac{2 \mu}{p+1} \| \psi (t) \|_{L^{p+1} (\R^d ; \C^2)}^{p+1} ,
\end{align}
respectively, where $T_{\mr{P}} [\psi, \v{A}] = \| \gvsig \cdot (\v{p} + \v{A}) \psi \|_{L^2 (\R^d ; \C^2)}^2$ is the total Pauli kinetic energy. Again, we usually surpress the $\v{A}$-dependence of $E_{\mr{P}}$ and $T_{\mr{P}}$. At least formally, the $L^2$-norm $\| \psi (t) \|_{L^2 (\R^d ; \C^2)}^2$ and the total energy (\ref{def:total_mag_energy_pauli}) are conversed along the flow generated by (\ref{eq:NLP}). 

Consider the case of a uniform magnetic field $\v{B} = (0,0,B)$, $B \in \R \backslash \{0\}$. Using the algebraic properties of the Pauli matrices we note that 
\begin{align*}
[\gvsig \cdot (\v{p} + \v{A})]^2 = (\v{p} + \v{A})^2 + B \sigma_3 .
\end{align*}
As a consequence, we have that the unitary time evolution for the Pauli operator $U_{\mr{P}} (t) = \exp{ \{ - i t [ \gvsig \cdot (\v{p} + \v{A}) ]^2 \} }$ is equal to
\begin{align*}
U_{\mr{P}} (t) = e^{- i B t \sigma_3} U_{\mr{S}} (t) ,
\end{align*}
where $U_{\mr{S}} (t)$ is given by (\ref{eq:unitary_propagator_schr_constant_field}). Hence, the estimates (\ref{eq:US-Lp_estimate}), (\ref{eq:homogeneous_Strichartz_S_constant_field_3D}), and (\ref{eq:nonhomogeneous_Strichartz_S_constant_field_3D}) continue to hold with $U_{\mr{S}} (t)$ replaced with $U_{\mr{P}} (t)$. Likewise, Theorem \ref{thm:strichartz_identity} easily generalizes to $U_{\mr{P}} (t)$ in $d = 2$ dimensions. Therefore, using the same proof that was used for Theorem \ref{thm:CE88}, we have the following Theorem.
\begin{thm}\label{thm:local_wellposed_Pauli}
Let $d \in \{2,3\}$, $\mu \in \R$, $1 < p < 1 + 4 / (d-2)$, and $\v{A} = B \v{x}_{\perp} / 2$. For all $\psi_0 \in H_{\!\v{A}}^1 (\R^d ; \C^2)$ we have the following.
\begin{enumerate}
\item There exists a unique maximal solution $\psi \in C([0,T_*) , H_{\!\v{A}}^1) \cap C^1 ([0,T_*) , H_{\!\v{A}}^{-1})$ of (\ref{eq:NLP}). If $T_* < \infty$, then $\| (\v{p} + \v{A}) \psi (t) \|_{L^2 (\R^d)} \rightarrow \infty$ at $t \uparrow T_*$.
\item The mapping $\psi_0 \mapsto T_* (\psi_0)$ is lower semi-continuous and, if $t \in [0 , T_* (\psi_0))$ and $( \phi_n )_{n \geq 1} \subset H^1_{\!\v{A}}$ converges to $\psi_0$ as $n \rightarrow \infty$, in $H_{\!\v{A}}^1$, then the corresponding sequence of solutions $(\psi_n)_{n \geq 1}$ to (\ref{eq:NLP}) verify $\psi_n \rightarrow \psi$ as $n \rightarrow \infty$, in $C([0,t] , H_{\!\v{A}}^1)$.
\item If $\psi_0 \in H^2_{\!\v{A}}$, then $\psi \in C([0,T_*) , H_{\!\v{A}}^2) \cap C^1 ([0,T_*) , L^2)$.
\item $\| \psi (t) \|_{L^2 (\R^d)} = \| \psi_0 \|_{L^2 (\R^d)}$ and $E_{\mr{P}} [\psi (t) , \v{A}] = E_{\mr{P}} [\psi_0, \v{A}]$. 
\end{enumerate}
\end{thm}

In general, the diamagnetic inequality (\ref{eq:diamagnetic}) no longer holds when $\v{p} + \v{A}$ is replaced by $\gvsig \cdot (\v{p} + \v{A})$. However, in $3$ dimensions, as a result of the estimate
\begin{align*}
\langle \gvsig \cdot \v{B} \psi , \psi \rangle_{L^2(\R^3 ; \C^2)} = B \int_{\R^3} \left( |\psi_1|^2 - |\psi_2|^2 \right) \leq B \| \psi \|_{L^2 (\R^3 ; \C^2)}^2 ,
\end{align*}
we may still obtain a uniform bound on $\| (\v{p} + \v{A}) \psi \|_{L^2 (\R^3 ; \C^2)}$ when $p < 7/3$ in a similar manner as the magnetic Schr\"{o}dinger case discussed in \S\ref{sub:strichartz}. A similar estimate in $2$ dimensions shows the same conclusion holds, but now with $p < 3$. Therefore, by Theorem \ref{thm:local_wellposed_Pauli} we conclude global well-posedness of the Cauchy problem (\ref{eq:NLP}) in the range $1 < p < 1 + 4/d$, $d \in \{2,3\}$, with a uniform magnetic field. 

Our blow-up result for (\ref{eq:NLP}) is very similar to that for the NLMS equation (\ref{eq:NLMS}). To state the result, we introduce the function space
\begin{align*}
\Sigma_{\mr{P}} := \{ f \in H^1_{\!\v{A}} (\R^d ; \C^2) ~ : ~ \int_{\R^d} |\v{x}|^2 |f (\v{x})|^2 \dd \v{x} < \infty \} ,
\end{align*}
and define $F_{\mr{P}} : \Sigma_{\mr{P}} \rightarrow \R$ by
\begin{align}\label{def:blowup_functional_Pauli}
F_{\mr{P}} [\psi] = E_{\mr{P}} [\psi] - B \re{ \langle \gvsig \cdot \v{x}_{\perp} \psi , \gvsig \cdot (\v{p} + \v{A}) \psi \rangle_{L^2(\R^d ; \C^2)} } + \frac{B^2}{2} \| \rho \psi \|_{L^2 (\R^d ; \C^2)}^2 .
\end{align}
Similar to $F_{\mr{S}}$, the next Lemma shows $F_{\mr{P}}$ is conserved under the time evolution of (\ref{eq:NLP}).
\begin{lem}\label{lem:conservation_of_Fp}
Let $d \in \{2,3\}$ and $\v{A} \in L_{\mr{loc}}^2 (\R^d ; \R^d)$ generate a uniform magnetic field $\v{B} = (0,0,B)$. Let $\psi_0 \in H_{\!\v{A}}^1 (\R^d ; \C^2)$ and $\psi \in C([0,T_*) , H_{\!\v{A}}^1) \cap C^1 ([0,T_*) , H_{\!\v{A}}^{-1})$ denote the corresponding solution to (\ref{eq:NLMS}). Then, $F_{\mr{P}} [\psi (t)] = F_{\mr{P}} [\psi_0]$. 
\end{lem}

The proof of Lemma \ref{lem:conservation_of_Fp} is almost identical to that of Lemma \ref{lem:conservation_of_Fs} and is reserved for the end of \S\ref{sub:virial}. Using that $F_{\mr{P}} [\psi]$ is conserved we have the following Theorem concerning the second time derivative of the expectation value of $|\v{x}|^2$.

\begin{thm}\label{thm:blowup_NLP}
Let $d \in \{ 2, 3 \}$, $1 < p < 1 + 4/(d-2)$, $\v{A} \in L_{\mr{loc}}^2 (\R^d ; \R^d)$ generate a uniform magnetic field $\v{B} = (0,0,B)$, and $\psi_0 \in \Sigma_{\mr{P}}$. Let $T_* \in (0 , \infty]$ be the time so that $\psi \in C ( [0, T_*) ; \Sigma_{\mr{P}} ) \cap C^1 ( [0, T_*) ; H_{\!\v{A}}^{-1} )$ is the corresponding maximal solution to the Cauchy problem (\ref{eq:NLP}). Then the function $g (t) = \frac{1}{4} \| \v{x} \psi (t) \|_{L^2 (\R^d ; \C^2)}^2$ satisfies the virial identity
\begin{align}\label{eq:constantfield_xsquare_gauge_independent_pauli}
\ddot{g} (t) = 2 F_{\mr{P}} [\psi_0] + \mu d \frac{p - (1 + 4 / d)}{p+1} \| \psi (t) \|_{L^{p+1} (\R^d)}^{p+1} - B^2 \| \rho \psi (t) \|_{L^2 (\R^d ; \C^2)}^2 .
\end{align}
In particular, if $d = 2$ and $p = 3$, then (\ref{eq:constantfield_xsquare_gauge_independent_pauli}) becomes a second-order equation for $g$ that can be solved exactly:
\begin{align}\label{eq:constantfield_explicitsolution_pauli}
g(t) = \frac{F_{\mr{P}} [\psi_0]}{2 B^2} + \left( g_0 - \frac{F_{\mr{P}} [\psi_0]}{2 B^2} \right) \cos{(2Bt)} + \frac{\dot{g}_0}{2 B} \sin{(2Bt)}
\end{align}
where $g_0 = \frac{1}{4} \| \rho \psi_0 \|_{L^2 (\R^2 ; \C^2)}^2$ and $\dot{g}_0 = \re{ \langle \v{x} \psi_0 , (\v{p} + \v{A}) \psi_0 \rangle_{L^2(\R^2 ; \C^2)}}$.
\end{thm}

\begin{cor}
Suppose $\mu < 0$, $1 + 4/d \leq p < 1 + 4/(d-2)$, $\psi_0 \in \Sigma_{\mr{P}}$, and $F_{\mr{P}} [\psi_0] < 0$. Then the corresponding solution to (\ref{eq:NLMS}) blows up in finite time. If $F_{\mr{P}} [\psi_0] = 0$, then blow-up occurs when $\dot{g} (0) = \re{ \langle \v{x} \psi_0 , (\v{p} + \v{A}) \psi_0 \rangle_{L^2 (\R^d)}} < 0$.
\end{cor}

As with the NLMS equation, both Theorem \ref{thm:blowup_NLP} and its Corollary are proved by deriving a virial identity for the second time derivative of $\langle \psi ,  |\v{x}|^2 \psi \rangle_{L^2(\R^d ; \C^2)}$. As before, we first treat the case of a general, time-independent, external magnetic field (not necessarily a uniform field).  We record the virial identity for the NLP equation as the following Lemma. 
\begin{lem}\label{lem:NLP_Virial}
Let $g (t) = \frac{1}{4} \| \v{x} \psi (t) \|_{L^2 (\R^d ; \C^2)}^2$. Then, for any solution $\psi$ to the NLP equation (\ref{eq:NLP}) with initial data $\psi_0 \in H_{\!\v{A}}^2 (\R^d ; \C^2)$, the following virial identities holds: In dimension $d = 3$,
\begin{align*}
\ddot{g} & = 2 T_{\mr{P}} [\psi] + 3 \mu \frac{p-1}{p+1} \| \psi \|_{L^{p+1} (\R^2 ; \C^2)}^{p+1} \\
& \hspace{3cm} + 2 \re{ \langle \gv{\sigma} \cdot (\v{x} \wedge \v{B}) \psi , \gvsig \cdot (\v{p} + \v{A}) \psi \rangle_{L^2(\R^2 ; \C^2)}} , \numberthis \label{eq:xsquared_3_dimension_Pauli} 
\end{align*}
and, assuming $\v{B}$ is aligned with the $x_3$-axis, in dimension $d = 2$,
\begin{align*}
\ddot{g} & = 2 T_{\mr{P}} [\psi] + 2 \mu \frac{p-1}{p+1} \| \psi \|_{L^{p+1} (\R^2 ; \C^2)}^{p+1} \\
& \hspace{3cm} - 2 \re{ \langle B \gv{\sigma} \cdot \v{x}_{\perp} \psi , \gvsig \cdot (\v{p} + \v{A}) \psi \rangle_{L^2(\R^2 ; \C^2)}}. \numberthis \label{eq:xsquared_2_dimension_Pauli}
\end{align*}
\end{lem}
\begin{proof}
We denote $H_{\mr{P}} =  [ \gvsig \cdot \gvpi]^2 + \mu | \psi|^{p-1}$ where $\gvpi = \v{p} + \v{A}$. For simplicity we focus on the $d = 3$ case, as the $d = 2$ case will be nearly identical. Computing first time derivative of $g$ is essentially the same as the first time derivative of $g$ in the proof of Lemma \ref{lem:NLMS_Virial}. For the second time derivative we find
\begin{align}\label{eq:varia_1_P}
\ddot{g} = \langle i [ H_{\mr{P}} , (D + \v{x} \cdot \v{A}) ] \psi , \psi \rangle_{L^2(\R^3 ; \C^2)} ,
\end{align}
where $H_{\mr{P}} = [\gvsig \cdot \gvpi]^2 + \mu |\psi|^{p-1} \psi$. We worked out the commutator $[H_{\mr{P}},D]$ in the same way as in the proof of Lemma \ref{lem:NLMS_Virial}. We find
\begin{align}\label{eq:varia_2_P}
i [H_{\mr{P}} , D]  = 2 [\gvsig \cdot \gvpi]^2 - \mu (\v{x} \cdot \nabla ) |\psi|^{p-1} + \frac{\dd}{\dd \theta} [\gvsig \cdot \gvpi_{\theta}]^2 \Big|_{\theta = 0} .
\end{align}
where $\gvpi_{\theta} = \v{p} + \v{A}_{\theta}$.

Likewise, for the computation of
\begin{align*}
\langle i [ H_{\mr{P}} , \v{x} \cdot \v{A} ] \psi , \psi \rangle_{L^2 (\R^3 ; \C^2)} = \langle i [ [ \gvsig \cdot \gvpi ]^2 , \v{x} \cdot \v{A} ] \psi , \psi \rangle_{L^2(\R^3 ; \C^2)}  = \langle i [ \gvpi^2 , \v{x} \cdot \v{A} ] \psi , \psi \rangle_{L^2(\R^3 ; \C^2)}
\end{align*}
we refer to the proof of Lemma \ref{lem:NLMS_Virial}. In total we arrive at
\begin{align*}
g & = 2 T_{\mr{P}} [\psi , \v{A}] + 3 \mu \frac{p-1}{p+1} \| \psi \|_{L^{p+1} (\R^d)}^{p+1} \\
& \hspace{5mm} + \frac{\dd }{\dd \theta} \langle \gvsig \cdot \v{B}_{\theta} \psi , \psi \rangle_{L^2(\R^d)} + 2 \re{ \langle \v{x} \wedge \v{B} \psi , \gvpi \psi \rangle_{L^2 (\R^d)}} . \numberthis \label{eq:varia_3_P}
\end{align*}
where $\v{B}_{\theta} (\v{x}) = e^{-2 \theta} \v{B} (e^{-\theta} \v{x})$. We may simplify the expression (\ref{eq:varia_3_P}) by observing the following calculation:
\begin{align*}
\int_{\R^3} (\v{x} \wedge \v{B}) \cdot \curl{\langle \psi , \gvsig \psi \rangle_{\C^2}} \dd \v{x} & = \int_{\R^3} \curl{(\v{x} \wedge \v{B})} \cdot \langle \psi , \gvsig \psi \rangle_{\C^2} \dd \v{x} \\
& = \int_{\R^3} (- 2 \v{B} - (\v{x} \cdot \nabla) \v{B}) \cdot \langle \psi , \gvsig \psi \rangle_{L^2(\R^3 ; \C^2)} \dd \v{x} \\
& = \frac{\dd }{\dd \theta} \Big|_{\theta = 0} \langle \gvsig \cdot \v{B}_{\theta} \psi , \psi \rangle_{L^2(\R^3 ; \C^2)} .
\end{align*}
Combining the previous observation with (\ref{eq:varia_3_P}) we arrive at (\ref{eq:xsquared_3_dimension_Pauli}).
\end{proof}

\begin{proof}[Proof of Lemma \ref{lem:conservation_of_Fp} and Theorem \ref{thm:blowup_NLP}]
The proof of Lemma \ref{lem:conservation_of_Fp} and, hence, Theorem \ref{thm:blowup_NLP} is nearly identical to the proof for the NLMS equation case. In particular, we have the identity
\begin{align*}
\frac{\dd }{\dd t} \re{ \langle \gvsig \cdot \v{x}_{\perp} \psi , \gvsig \cdot (\v{p} + \v{A}) \psi \rangle_{L^2(\R^d)}} = \frac{B}{2} \frac{\dd }{\dd t} \| \rho \psi \|_{L^2 (\R^d)}^2 ,
\end{align*}
which upon integration in time together with (\ref{eq:xsquared_2_dimension_Pauli})-(\ref{eq:xsquared_3_dimension_Pauli}) yields the desired result.
\end{proof}

\end{document}